\documentclass[12pt]{amsart}
\usepackage[?]{amsrefs}

\newtheorem{thm}{Theorem}[section]
\newtheorem{lem}[thm]{Lemma}
\newtheorem{prop}[thm]{Proposition}

\begin{document}
\baselineskip=17pt

\title{Orbits in symmetric spaces, II}

\author{N. J. Kalton}
\address{Department of Mathematics, University of Mis\-souri-Columbia, Colum\-bia, Missouri 65211, U.S.A. }
\email{kaltonn@missouri.edu}

\author{F. A. Sukochev}
\address{School of Mathematics and Statistics, University of New South Wales, Sydney, NSW 2052, Australia }
\email{f.sukochev@unsw.edu.au}

\author{D. Zanin}

\address{School of Computer Science, Engineering and Mathematics, Flinders University, Adelaide, SA 5042, Australia}
\email{zani0005@csem.flinders.edu.au}



\thanks{The first author acknowledges support from NSF grant DMS-0555670; the second and third authors acknowledge support from the ARC}

\begin{abstract}  Suppose $E$ is fully symmetric Banach function space on $(0,1)$ or $(0,\infty)$ or a fully symmetric Banach sequence space.  We give necessary and sufficient conditions on $f\in E$ so that its orbit $\Omega(f)$ is the closed convex hull of its extreme points. We also give an application to symmetrically normed ideals of compact operators on a Hilbert space.
\end{abstract}

\maketitle

\section{Introduction}
Let $I$ be either the interval $(0,1)$ or the semi-axis $(0,\infty)$ and suppose $f\in L_1(I)+L_\infty(I).$  We define the {\it orbit} $\Omega(f)$ of $f$ to be set of $Tf$ where $T:L_1+L_\infty\to L_1+L_\infty$ is an operator with $\|T\|_{L_1\to L_1},\|T\|_{L_\infty\to L_\infty}\le 1$ (see \cites{Ryff1965, KreinPetuninSemenov1982}).  Then it follows from the Calder\'on-Mitjagin Theorem \cites{Calderon1966,Mitjagin1965, BennettSharpley1982, KreinPetuninSemenov1982} that $\Omega(f)$ can be characterized as the set of $g\in L_1+L_\infty$ such that
\begin{equation}\label{HLP}
\int_0^t g^*(s)\,ds\le \int_0^t f^*(s)\,ds, \qquad 0<t<\infty
\end{equation}
where as usual $f^*$ is the decreasing rearrangement of $|f|$ (see \S \ref{prelim} for definitions).  This may be written $g\preceq f$ where $\preceq$ is the Hardy-Littlewood-Polya ordering.
Thus $E$ is an exact interpolation space if and only if it is fully symmetric (see \S \ref{prelim}).

The extreme points of $\Omega(f)$, which we denote $\partial_e\Omega(f)$ were obtained in \cite{Ryff1967} (for case of spaces on $(0,1)$) and \cite{ChilinKryginSukochev1992} (for the general case).  Except in the special case when $I=(0,\infty)$ and $E\supset L_\infty$ these are given by $\partial_e\Omega(f)=\{g:\ g^*=f^*\}$ (see \S \ref{prelim} for full details; in the exceptional cases the extreme points from a subset of this set).
Let $\mathcal Q(f)$ be the convex hull of the set $\{g:\ g^*\le f^*\}.$  Then it is clear that if $E$ is fully symmetric and $f\in E$ the closure $\mathcal Q_E(f)$ of $\mathcal Q(f)$ in $E$ coincides with the closed convex hull of $\partial_e\Omega(f).$

In \cite{Ryff1965} it was shown for the case of $I=(0,1)$ that the orbit $\Omega(f)$ is always weakly compact in $L_1(0,1).$  It follows from results in \cite{DoddsSukochevSchluchtermann2001} that if $E$ is an order-continuous (equivalently, separable) symmetric function space (which is necessarily fully symmetric) and $f\in E$ then $\Omega(f)$ is weakly compact in $E$.  Thus it is an immediate consequence of the Krein-Milman theorem that  $\mathcal Q_E(f)$ coincides with $\Omega(f).$

For the case of non-separable fully symmetric spaces the situation is less clear.  The example $E=L_\infty$ and $f=1$ shows that $\mathcal Q_E(f)$ may still coincide with $\Omega(f).$  This problem was first investigated by Braverman and Mekler \cite{BravermanMekler1977} for the unit interval i.e. $I=(0,1)$.  They gave a sufficient condition for $\Omega(f)=\mathcal Q_E(f)$ in terms of the behavior of the dilation operators $\sigma_\tau$ (see \S \ref{prelim} for the appropriate definitions).  Precisely they showed that if $E$ is a fully symmetric Banach function space on $(0,1)$ such that
$$\lim_{\tau\to\infty}\frac{\|\sigma_\tau\|_{E\to E}}{\tau}=0$$
then $\Omega(f)=\mathcal Q_E(f)$ for every $f\in E.$  This condition is, however, not necessary since it may fail in separable symmetric spaces (e.g. $E=L_1$).

Recently two of the current authors \cite{SukochevZanin2009} found a necessary and sufficient condition for the similar problem concerning the positive part of the orbit.  If $f\ge 0$ we denote by $\Omega_+(f)$ the set $\{g:\ g\in\Omega(f),\ g\ge 0\}.$  In \cite{SukochevZanin2009} it was shown that for a fully symmetric Banach function space $E$ with a Fatou norm (sometimes called a weak Fatou property) that if $f\in E_+$ then $\Omega_+(f)$ coincides with the closed convex hull of its extreme points if and only if a local Braverman-Mekler type condition holds.  If $I=(0,1)$ or if $I=(0,\infty)$ and $E$ is not contained in $L_1(0,\infty)$ this condition takes the form
\begin{equation}\label{SZ1}
\lim_{\tau\to\infty} \frac{\|\sigma_\tau(f^*)\|_E}{\tau}=0.
\end{equation}
If $I=(0,\infty)$ and $E\subset L_1$ we must replace \eqref{SZ1} by
\begin{equation}\label{SZ2}
\lim_{\tau\to\infty} \frac{\|\chi_{(0,1)}\sigma_\tau(f^*)\|_E}{\tau}=0.
\end{equation}
The results of \cite{SukochevZanin2009} imply that under the same hypotheses on $E$ (full symmetricity and a Fatou norm) that \eqref{SZ1} and \eqref{SZ2} are sufficient for $\mathcal Q_E(f)=\Omega(f).$

Our main result in this paper is to show that, indeed, if $E$ is a fully symmetric Banach space with a Fatou norm on $(0,1)$ or $(0,\infty)$, \eqref{SZ1} and \eqref{SZ2} are necessary and sufficient for $\Omega(f)=\mathcal Q_E(f).$  These results are Theorems \ref{Main3}, \ref{Main4} and \ref{Main30} below.  We also establish the corresponding result for sequence spaces in Theorem \ref{Main5}; sequence spaces were not covered in \cite{SukochevZanin2009} so we are also able to complete the picture for the positive part of the orbit.

We conclude the paper with an application to orbits in symmetrically normed ideals of compact operators on a Hilbert space.

\section{Preliminaries}\label{prelim}

In this section we present some definitions from the theory of symmetric spaces. For more details on the latter theory we refer to \cites{KreinPetuninSemenov1982, LindenstraussTzafriri1979, BennettSharpley1982}.

Let $I$ denote either $(0,1)$ or on $(0,\infty)$ with Lebesgue measure $\mu$. If $f\in L_1(I)+L_\infty(I)$ we denote by $f^*$ the decreasing rearrangement of $f$, i.e.
$$ f^*(t)=\inf_{\mu A=t}\sup_{s\in I\setminus A}|f(s)|.$$
If $f,g$ are functions in $L_1+L_\infty$ we write
$g\preceq f$ if
$$ \int_0^t g^*(s)\,ds \le \int_0^t f^*(s)\,ds, \qquad t\in I.$$
This defines the Hardy-Littlewood-Polya ordering.

 A {\it symmetric Banach function space} $E$ on $I$ is a linear space with $L_1\cap L_\infty\subset E\subset L_1+L_\infty$, with an associated norm $\|\cdot\|_E$ such that $(E,\|\cdot\|_E)$ is complete and
 if $f\in E,\ g\in L_1+L_\infty$ with $g^*\le f^*$ then $g\in E$ and $\|g\|_E\le \|f\|_E.$   We will use $E_+$ to denote the positive cone of $E$ i.e. $\{f:\ f\in E,\ f\ge 0 \text{ a.e.}\}.$ We will also assume the normalization that $\|\chi_{(0,1)}\|=1.$  Let $\varphi_E(t)=\|\chi_{(0,t)}\|_E$ be the {\it fundamental function} of $E.$

$E$ is said to have a {\it Fatou norm} if for every sequence $(f_n)_{n=1}^{\infty}$ of nonnegative functions such that $f_n\uparrow f$ a.e. with $f\in E$ we have $\lim_{n\to\infty}\|f_n\|_E=\|f\|_E.$

A symmetric Banach function space $E$ is said to be {\it fully symmetric} if and only if $f\in E,\ g\in L_1+L_{\infty}$ with $g\preceq f$, then $g\in E$ and $||f||_E\leq ||g||_E.$   $E$ is fully symmetric precisely when $E$ is an exact interpolation space for the couple $(L_\infty(I),L_1(I))$ by the Calder\'on-Mitjagin theorem \cites{Mitjagin1965, Calderon1966}.  In this paper we will only consider fully symmetric Banach function spaces.

We will need the following inequality can be found in \cite{KreinPetuninSemenov1982}, Theorem II.3.1. If $f,g\in L_1+L_{\infty},$ then
\begin{equation}\label{second}
(f^*-g^*)\preceq (f-g)^*.
\end{equation}
As a consequence if $E$ is fully symmetric and $f,g\in E$ we have
\begin{equation}\label{KPS}
\|f^*-g^*\|_E\le \|f-g\|_E.
\end{equation}

If $E$ is a fully symmetric Banach function space and $f\in E$ we define the orbit of $f$ by $\Omega(f)=\{g:\ g^*\preceq f^*\}\subset E.$
The set of the extreme points of the set $\Omega(f)$ is well-known (see \cites{Ryff1967,ChilinKryginSukochev1992}) and, if $I=(0,1)$ or $I=(0,\infty)$, and $E$ does not contain  $L_\infty$ it is given by
$$\partial_e(\Omega(f))=\{g\in L_1+L_{\infty}:\ f^*=g^*\}.$$  If $I=(0,\infty)$ and $E$ contains $L_\infty$ we must make a small correction:
$$\partial_e(\Omega(f))=\{g\in L_1+L_{\infty}:\ f^*=g^*,\ |g(t)|\ge \lim_{s\to\infty}f^*(s) \text{ a.e. }\}.$$
We define $\mathcal Q(f)$ to be the convex hull of the set
$\{g\in L_1+L_{\infty}:\ g^*\le f^*\}.$ We will denote by
$\mathcal Q_E(f)$ the closure in $E$ of $\mathcal Q(f).$  This
is easily seen to coincide with the closed convex hull of
$\partial_e\Omega(f).$ Thus $\mathcal Q_E(f)\subset \Omega(f).$

We next define the dilation operators on $E$.
If $\tau>0$ and $I=(0,\infty)$  the dilation operator $\sigma_{\tau}$ is defined by setting $$(\sigma_{\tau}(f))(s)=f({s}/{\tau}),\qquad s>0.$$  In the case of the interval $(0,1)$ the operator $\sigma_{\tau}$ is defined by
$$
(\sigma_{\tau}f)(s)=
\begin{cases}
f(s/\tau),& s\leq\min\{1,\tau\}\\
0,& \tau<s\leq1.
\end{cases}
$$
The operators $\sigma_{\tau}$ ($\tau\geq1$) satisfy semi-group property $\sigma_{\tau_1}\sigma_{\tau_2}=\sigma_{\tau_1\tau_2}.$ If $E$ is a symmetric space and if $\tau>0,$ then the dilation operator $\sigma_{\tau}$ is a bounded operator on $E$ and
$$||\sigma_{\tau}||_{E\to E}\leq\max\{1,\tau\}.$$

  If $E$ is a fully symmetric function space on $(0,\infty)$ then $E+L_\infty$ is also a fully symmetric function space under the norm
$$ \|f\|_{E+L_\infty}=\|f^*\chi_{(0,1)}\|_E.$$  The next Lemma will be used later.

\begin{lem} \label{L1} Let $E$ be a symmetric function space on $(0,\infty)$, such that $E\setminus L_1\neq \emptyset,$ and suppose $f\in L_1\cap E.$  Then
$$ \lim_{\tau\to\infty}\tau^{-1}\|\sigma_\tau(f)\|_E=\lim_{\tau\to\infty}\tau^{-1}\|\sigma_\tau(f)\|_{E+L_\infty}.$$
\end{lem}

\begin{proof} We may suppose $f$ is nonnegative and decreasing. Let $\varphi=\varphi_E$ be the fundamental function of $E$ and let $\psi$ be its least concave majorant of $\varphi.$ Since $E\setminus L_1\neq\emptyset$ we have $\lim_{t\to\infty}\psi'(t)=0.$ For any $\tau>1$ we have, using Theorem II.5.5 of \cite{KreinPetuninSemenov1982}
\begin{align*} \|(\sigma_{\tau}f)\chi_{(1,\infty)}\|_E &\le \|f(\tau^{-1})\chi_{(0,1)}+(\sigma_\tau f)\chi_{(1,\infty)}\|_E\\
&\le f(\tau^{-1})\int_0^1\psi'(s)\,ds +\int_1^\infty \psi'(s) f(\tau^{-1}s)\,ds\\
& \le \psi(1)f(\tau^{-1})+ \tau\int_{\tau^{-1}}^{\infty}\psi'(\tau s)f(s)\,ds.\end{align*}

Now we have, since $f\in L_1,$
$$ \lim_{\tau\to\infty}\tau^{-1}f(\tau^{-1})=0$$ and by the Dominated Convergence Theorem,
$$ \lim_{\tau\to\infty}\int_{\tau^{-1}}^{\infty}\psi'(\tau s)f(s)\,ds=\lim_{\tau\to\infty}\int_{0}^{\infty}\chi_{(\tau^{-1},\infty)}(s)\psi'(\tau s)f(s)\,ds=0.$$  Hence
$$ \lim_{\tau\to\infty}\tau^{-1}\|(\sigma_{\tau}f)\chi_{(1,\infty)}\|_E=0$$ and the Lemma follows.
\end{proof}

We next discuss the corresponding notions for sequence spaces.  If $\xi=(\xi_n)_{n=1}^{\infty}$ is a sequence then $\xi^*$ denotes its decreasing rearrangement:
$$\xi^*_n=\inf_{|\mathbb A|=n-1}\sup_{k\in\mathbb N\setminus \mathbb A}|\xi_k|.$$
A Banach sequence space $E$ is called {\it symmetric} if $\xi\in E$ and $\eta^*\le \xi^*$ implies that $\eta\in E$ and $\|\eta\|_E\le \|\xi\|_E.$
We write $\eta\preceq \xi$ if
$$ \sum_{k=1}^n\eta_k^*\le \sum_{k=1}^n\xi_k^*, \qquad n\in\mathbb N.$$  $E$ is called {\it fully symmetric} if $\xi\in E$ and $\eta\preceq \xi$ implies that $\eta\in E$ and $\|\eta\|_E\le \|\xi\|_E.$
If $\xi$ is any bounded sequence we define its orbit $\Omega(\xi)=\{\eta:\, \eta\preceq \xi\}.$

In this context, we define the dilation operators $\sigma_m$ only for $m\in\mathbb N.$  Then
$$ \sigma_m(\xi)=(\xi_1,\ldots,\xi_1,\xi_2,\ldots,\xi_2,\xi_3\ldots)$$ where each $\xi_j$ is repeated $m$ times.

\section{Approximation of the orbit}

Our first proposition gives a simple criterion which will enable us to check when $\mathcal Q_E(f)=\Omega(f).$

\begin{prop}\label{convexapprox} Let $E$ be a fully symmetric Banach space on $(0,\infty).$  Suppose $f,g$ are  nonnegative decreasing functions in $E$.  Then $g\in \mathcal Q_E(f)$ if and only if, given $\epsilon>0,$ there exists a nonnegative decreasing function $h\in E$ and an integer $p$ such that $0\le h\le g$  and
\begin{equation}\label{20}
\|g-h\|_E<\epsilon
\end{equation}
and
\begin{equation}\label{21}
\int_{pa}^b h(t)\,dt\le \int_a^b f(t)\,dt, \qquad 0<pa<b<\infty.
\end{equation}
\end{prop}

\begin{proof} Suppose first $g\in\mathcal Q_E(f).$  Then given $\epsilon>0$ there exist $f_1,\ldots,f_p\in E$ such that
$f_j^*\le f$ for $1\le j\le p$ and
$$ \|g-\frac1p(f_1+\cdots+f_p)\|_E<\epsilon.$$
Let
$$ u=\frac1p(f_1+\cdots+f_p),\quad v=\frac1p(|f_1|+\cdots+|f_p|).$$
Then if $h=g\wedge v^*,$ using \eqref{KPS}
$$ \|g- h\|_E\le \|g-g\wedge u^*\|_E\le \|g-u^*\|_E\le \|g-u\|_E<\epsilon.$$
It remains to observe that \eqref{21} holds by Lemma 4.1 of \cite{KaltonSukochev2008}.

The converse  is easy.  If $h$ satisfies \eqref{20} and \eqref{21} then $h\in \alpha \mathcal Q(f)$ for every $\alpha>1$ by Theorem 6.3 of \cite{KaltonSukochev2008}.  Hence $h\in\mathcal Q_E(f)$   and so $d(g,\mathcal Q_E(f))<\epsilon.$
\end{proof}

The next Lemma is surely well-known but we use it in the main result and include a proof for completeness.

\begin{lem} \label{concaveminorant} Let $F$ be a continuous nonnegative increasing concave function on $[0,\infty)$ with $F(0)=0.$ Let us suppose that $(\alpha_n)_{n\in\mathbb Z}$ is an increasing doubly infinite sequence of distinct positive reals with $$\lim_{n\to-\infty}\alpha_n=0, \qquad\lim_{n\to\infty}\alpha_n=\infty.$$  Suppose that $(\beta_n)_{n\in\mathbb Z}$ is any sequence with $$0\le \beta_n\le F(\alpha_n), \qquad n\in \mathbb Z.$$  (i) There is a least concave function $G$ on $
[0,\infty)$ such that $G(0)\ge 0,$ and $G(\alpha_n)\ge \beta_n$ for $n\in\mathbb Z.$   $G$ is continuous nonnegative and increasing and $G(0)=0.$
\newline
(ii) Furthermore if $n\in\mathbb Z$ then either
$$ G(t)=G(\alpha_n)t/\alpha_n, \qquad 0\le t\le \alpha_n$$ or
there exists $m<n$ so that
$$ G(t)=\beta_m + \frac{G(\alpha_n)-\beta_m}{\alpha_n-\alpha_m}(t-\alpha_m), \qquad \alpha_m\le t\le \alpha_n.$$
\end{lem}

\begin{proof}  (i) is almost immediate.  $G$ is defined as the infimum of the collection $\mathcal C$ of all increasing concave functions $H$ on $[0,\infty)$ such that $H(\alpha_n)\ge \beta_n$ for all $n\in\mathbb Z$ and $H(0)\ge 0.$  This collection is non-empty since $F\in\mathcal C.$  $G$ is affine on each interval $[\alpha_n,\beta_{n+1}]$ and  since $G\le F$, $G$ is continuous at $0.$

For (ii), assume $G$ is not affine on $[0,\alpha_n]$.  Then there exists a least $p<n$ so that $g$ is affine on $[\alpha_{p},\alpha_n].$  Let $G_0$ be the function equal to $G$ on $[0,\alpha_{p-1}]$ and $[\alpha_n,\infty)$ and affine on $[\alpha_{p-1},\alpha_n].$
Then for any $0<\lambda<1$ we have $(1-\lambda)G+\lambda G_0\notin \mathcal C.$  Hence there exists $k(\lambda)\in\{p,p+1,\ldots,n-1\}$ so that
$$ (1-\lambda)G(\alpha_{k(\lambda)})+\lambda G_0(\alpha_{k(\lambda)})<\beta_{k(\lambda)}.$$
Letting $\lambda\to 0$ through as a suitable sequence where $k(\lambda)=m<n$ is constant we obtain
$ G(\alpha_m)=\beta_m$ and the second alternative holds.\end{proof}

We now prove our main result.

\begin{thm}\label{main1}  Let $E$ be a fully symmetric Banach function space on $(0,\infty)$ with  Fatou norm. Suppose $f\in E_+\setminus L_1$ is such that
$\Omega(f)=\mathcal Q_E(f).$  Then
$$ \lim_{\tau\to\infty}\tau^{-1}\|\sigma_\tau(f^*)\|_E=0.$$
\end{thm}

\begin{proof}  We may suppose that $f$ is decreasing.  We let
$$ F(t)=\int_0^tf(s)\,ds.$$ Let us define a doubly infinite sequence $(a_n)_{n\in\mathbb Z}$ by $F(a_n)=(5/4)^n.$

We next introduce the family $\mathcal K$ of doubly infinite sequences $\kappa=(\kappa_n)_{n\in\mathbb Z}$ such that either $\kappa_n\in \mathbb N$ with $1\le \kappa_n< a_{n+1}/a_n$ or $\kappa_n=\infty.$  Then $\mathcal K$ is a complete lattice under the order $\kappa\le \kappa'$ if $\kappa_n\le \kappa'_n$ for all $n$.  We may define  the lattice operations $(\kappa\vee\kappa')_n=\max(\kappa_n,\kappa'_n)$ and $(\kappa\wedge\kappa')_n=\min(\kappa_n,\kappa'_n).$

For each $\kappa\in \mathcal K$ we define $\psi_{\kappa}\in E$ as follows.  Let $\Psi(t)=\Psi_{\kappa}(t)$ be the least increasing concave function such that $\Psi(0)\ge 0,$
$$\Psi(\kappa_na_n)\ge F(a_n), \qquad \text{ if } \kappa_n<\infty,$$ and
   $$ \Psi(a_{n})\ge 0 \qquad \text{ if }\kappa_n=\infty.$$ The existence and properties of $\Psi$ are guaranteed by applying Lemma \ref{concaveminorant} when  $\alpha_n=\kappa_na_n$ if $\kappa_n<\infty$ and $\alpha_n=a_{n}$ if $\kappa_n=\infty$ and $\beta_n=F(a_n)$ if $\kappa_n<\infty$ and $\beta_n=0$ if $\kappa_n=\infty.$ Since $F(a_n)\le F(\kappa_na_n)$ it is clear from Lemma \ref{concaveminorant} that $\Psi$ exists and $\Psi\le F.$ Furthermore $\Psi$ is piecewise affine on $(0,\infty)$ and we may define $\psi_{\kappa}=\Psi'$ is a nonnegative piecewise constant decreasing function on $(0,\infty)$.  Clearly $\psi_{\kappa}\in\Omega(f)\subset E.$

We note some elementary properties of the map $\kappa\to\psi_{\kappa}.$

\begin{lem}\label{elem}
\begin{equation}\label{11}
\psi_{\kappa}\preceq \psi_{\kappa'}, \qquad \text{ if } \kappa'\le \kappa
\end{equation}
\begin{equation}\label{12}
\psi_{\kappa\wedge\kappa'}\preceq \psi_{\kappa}\vee\psi_{\kappa'}, \qquad \kappa,\kappa'\in \mathcal K.
\end{equation}
\end{lem}

\begin{proof} \eqref{11} is quite trivial.

To see \eqref{12} note that
$$ \int_0^t\max(\psi_{\kappa}(s),\psi_{\kappa'}(s))\,ds\ge \max(\Psi_{\kappa}(t),\Psi_{\kappa'}(t)).$$  Now  if $\kappa_n\wedge\kappa'_n<\infty$ and $\kappa_n\le \kappa'_n$  we have
$$ \int_0^{\kappa_na_n}\max(\psi_{\kappa}(s),\psi_{\kappa'}(s))\,ds\ge \Psi_{\kappa}(\kappa_na_n)\ge F(a_n)$$ and with a similar inequality when $\kappa'_n<\kappa_n$ we obtain, from the definition of $\Psi_{\kappa\wedge\kappa'}$,
$$ \int_0^t\max(\psi_{\kappa}(s),\psi_{\kappa'}(s))\,ds\ge \Psi_{\kappa\wedge\kappa'}(t), \qquad 0\le t<\infty.$$
This proves \eqref{12}.
\end{proof}

\begin{lem}\label{technical}  Suppose $\kappa\in\mathcal K$ satisfies
\begin{equation}\label{assumption}
\max(\kappa_n,\kappa_{n+1})=\infty, \qquad n\in\mathbb{Z}.
\end{equation}
Then for any $n\in\mathbb Z$ such that $\kappa_n<\infty$ we have:
\begin{equation}\label{212}
\psi_{\kappa}(t)\ge \frac{9F(a_n)}{25\kappa_na_n} \qquad a_n\le t\le\kappa_na_n.
\end{equation}
\end{lem}

\begin{proof} If $f$ is not identically zero then $\psi_{\kappa}$ is only identically zero when $\kappa$ is identiclaly $\infty;$ we exclude this case so that $\Psi_{\kappa}(t)>0$ for $t>0.$
Observe first that $\psi_{\kappa}$ is constant on $(a_{n},\kappa_na_n).$  If for every $m<n$ such that $\kappa_m<\infty$  we have $\Psi_{\kappa}(\kappa_ma_m)>F(a_m)$ then
$$ \psi_{\kappa}(t)\ge \frac{F(a_n)}{\kappa_na_n}, \qquad 0<t\le \kappa_na_n.$$
 Otherwise, since $\Psi_{\kappa}(t)>0$ for all $t>0,$ we have that, by Lemma \ref{concaveminorant},  there exists  $m<n$ so that $\kappa_m<\infty$ and
$$\psi_{\kappa}(t)=\frac{\Psi_{\kappa}(\kappa_na_n)-F(a_m)}{\kappa_na_n-\kappa_ma_m}, \qquad \kappa_ma_m<t<\kappa_na_n.$$  Then
$$\psi_{\kappa}(t) \ge \frac{F(a_n)-F(a_m)}{\kappa_na_n-\kappa_ma_m}, \qquad a_n\le t\le\kappa_na_n.$$
Noting that $m\le n-2$ by \eqref{assumption} so that $F(a_m)\le (4/5)^2F(a_n)$  this implies that \eqref{212} holds for either alternative.
\end{proof}

For $\kappa\in\mathcal K$ and $r\in\mathbb N$ we will define a $\kappa^{[r]}\ge \kappa$ by suppressing the values of $\kappa$ which are less than $r$.  Precisely:
$$ \kappa^{[r]}_n=\begin{cases}
\kappa_n, \qquad \text{if } \kappa_n\ge r\\
\infty \qquad \text{ if } \kappa_n< r.
\end{cases}
$$

We next prove the following Lemma, which is the heart of the argument for Theorem \ref{main1}:

\begin{lem}\label{crucial}  Under the hypotheses of the theorem, we have that for any $\kappa\in\mathcal K$
$$ \lim_{r\to\infty}\|\psi_{{\kappa}^{[r]}}\|_E=0.$$
\end{lem}

\begin{proof} We will first prove the Lemma under the additional assumption that \eqref{assumption} holds.
  Since $\|\psi_{{\kappa}^{[r]}}\|_E$ is decreasing in $r$ (by \eqref{11}) it suffices to show that for given $\epsilon>0$ we can find $r$ so that $\|\psi_{{\kappa}^{[r]}}\|_E<\epsilon.$ By Proposition \ref{convexapprox} for any $\epsilon>0$ we can find a nonnegative decreasing function $h\le \psi_{\kappa}$ and an integer $p$ so that
\begin{equation}\label{210}
\int_{pa}^b h(t)\,dt\le \int_a^b f(t)\,dt, \qquad 0<pa<b<\infty,
\end{equation}
and
\begin{equation}\label{211}
\|\psi_{\kappa}-h\|_E<\epsilon/10.
\end{equation}

We shall take $r=36p.$   Let $v=10 (\psi_{\kappa}- h).$  We will show that $\psi_{\kappa^{[r]}}\preceq v.$  In order to do this we must show that if $\kappa_n^{[r]}<\infty$ we have
\begin{equation}\label{need}
F(a_n)\le \int_0^{\kappa_na_n}v^*(t)\,dt.
\end{equation}

If $\kappa_n^{[r]}<\infty$
\begin{align*} \int_0^{\kappa_na_n}v^*(t)\,dt &\ge \int_{pa_n}^{\kappa_na_n}v(t)\,dt\\ &=10 \left(\int_{pa_n}^{\kappa_na_n}\psi_{\kappa}(t)\,dt -\int_{pa_n}^{\kappa_na_n}h(t)\,dt\right)\\
&\ge 10 \left(\int_{pa_n}^{\kappa_na_n}\psi_{\kappa}(t)\,dt -\int_{a_n}^{\kappa_na_n}f(t)\,dt\right),\end{align*} by \eqref{211}.
Hence
by \eqref{212} of Lemma \ref{technical}
\begin{align*} \int_0^{\kappa_na_n}v^*(t)\,dt &\ge 10 \left(\frac{9(\kappa_na_n-pa_n)F(a_n)}{25\kappa_na_n}-\int_{a_n}^{a_{n+1}}f(t)\,dt\right)\\
&\ge 10\left( \frac{35}{36}\frac{9}{25}F(a_n)- \frac{1}4F(a_n)\right)\\
&= F(a_n).\end{align*}

This show that \eqref{need} holds and so $\psi_{\kappa^{[r]}}\preceq v$ and $\|\psi_{\kappa^{[r]}}\|_E<\epsilon.$
This completes the proof when \eqref{assumption} holds.

For the general case let us introduce $\kappa(0)_n=\kappa_n$ if $n$ is even and $\kappa(0)_n=\infty$ if $n$ is odd.  Similarly $\kappa(1)_n=\kappa_n$ if $n$ is odd and $\kappa(1)_n=\infty$ if $n$ is even. Both $\kappa(0)$ and $\kappa(1)$ satisfy \eqref{assumption} Then for an arbitrary $\kappa$ we have
$\kappa^{[r]}=\kappa(0)^{[r]}\wedge \kappa(1)^{[r]}$ and so by \eqref{12},
$$ \limsup_{r\to\infty}\|\psi_{\kappa^{[r]}}\|_E \le \limsup_{r\to\infty}\|\psi_{\kappa(0)^{[r]}}\|_E+\limsup_{r\to\infty}\|\psi_{\kappa(1)^{[r]}}\|_E=0.$$
\end{proof}

Next for any integer $p$ we define
$\gamma^{p}_n=p$ if $pa_n< a_{n+1}$ and $\gamma^{p}_n=\infty$ otherwise.  For each $q>p$ we define
$\gamma^{p,q}_n=p$ if $pa_n< a_{n+1}$ and $|n|\le q $ and $\gamma^{p,q}_n=\infty$ otherwise.
Let $\psi_p=\psi_{\gamma^{p}}$ and $\psi_{p,q}=\psi_{\gamma^{p,q}}.$

\begin{lem}\label{nextcrucial}  Under the hypotheses of the Theorem,
$$ \lim_{p\to\infty}\|\psi_p\|_E=0.$$
\end{lem}

\begin{proof}  Clearly $\|\psi_p\|_E$ is decreasing in $p$.  Assume $\|\psi_p\|_E>\epsilon>0$ for all $p\in\mathbb N.$
Since $E$ has a Fatou norm for each $p$ there exists $q(p)>p$ so that $\|\psi_{p,q(p)}\|_E>\epsilon.$  Let
$$ \kappa=\wedge_p\gamma^{p,q(p)}.$$  Thus $\kappa$ is given by the formula
$$ \kappa_n=\inf\{p: \ p<a_{n+1}/a_n,\ |n|\le q(p)\}$$ and $\kappa$ has the properties that $\kappa\le \gamma_{p,q(p)}$ for all $p$ and $\lim_{|n|\to\infty}\kappa_n=\infty.$

By Lemma \ref{crucial} there exists $r\in\mathbb N$ so that $\|\psi_{\kappa^{[r]}}\|_E<\epsilon.$  But then the set $\{n:\kappa_n<r\}$ is finite and so there is a choice of $p$ such that $p>a_{n+1}/a_n$ whenever $\kappa_n<r.$  Thus $\gamma^{p}_n=\infty$ if $\kappa_n<r.$  Thus
$$ \kappa^{[r]}\le \gamma^{p,q(p)}$$ and so by \eqref{11},
$$ \|\psi_{p,q(p)}\|_E <\epsilon$$ which gives a contradiction.
\end{proof}

We now can complete the proof of the Theorem.
We will show that if $p\in\mathbb N,$
\begin{equation}\label{311}
F(t) \le \frac45F(p^2t) +\frac54 \int_0^{p^2t}\psi_p(s)\,ds, \qquad 0<t<\infty.
\end{equation}
Indeed if \eqref{311} fails for some $t$ we can assume $a_n\le t<a_{n+1}$ for some $n\in\mathbb Z.$ We first argue that $a_{n+1}\le pa_n$.  Suppose, on the contrary, that $a_{n+1}>pa_n$. Then we have
$$\frac54\int_0^{p^2t}\psi_p(s)\,ds\ge \frac54 \int_0^{pa_n}\psi_p(s)\,ds\ge \frac54F(a_n)=F(a_{n+1})\ge F(t),$$ which contradicts our hypothesis. Next we show that $a_{n+2}\le pa_{n+1}.$ Indeed if $a_{n+2}>pa_{n+1},$ then $p^2t\geq pa_{n+1}$ and
$$\frac54\int_0^{p^2t}\psi_p(s)\,ds\ge \frac54\int_0^{pa_{n+1}}\psi_p(s)\,ds\ge \frac54 F(a_{n+1})>F(t).$$
But then $a_{n+2}\le p^2a_n$ and so
$$\frac45F(p^2t)\ge \frac45F(a_{n+2})=F(a_{n+1})>F(t)$$
and we have a contradiction. This establishes \eqref{311}.

Now if $\tau\ge 1$ we replace $t$ in \eqref{311} by $t/\tau$ and interpret the inequality in the form:
$$ \frac1\tau\sigma_\tau f \preceq \frac{1}{p^{-2}\tau}\sigma_{p^{-2}\tau}((4/5)f+(5/4)\psi_p).$$  Hence
$$\lim_{\tau\to\infty}\tau^{-1}\|\sigma_{\tau}f\|_E\le (4/5) \lim_{\tau\to\infty}\tau^{-1}\|\sigma_{\tau}f\|_E +(5/4)\lim_{\tau\to\infty}\tau^{-1}\|\sigma_\tau\psi_p\|_E$$ so that
$$ \lim_{\tau\to\infty}\tau^{-1}\|\sigma_{\tau}f\|_E\le (5^2/4)\|\psi_p\|_E.$$  Combining with Lemma \ref{nextcrucial} we obtain the theorem.
\end{proof}

The case when $f\in L_1$ is handled by reduction to the previous case:

\begin{thm}\label{main2}  Let $E$ be a fully symmetric Banach function space on $(0,\infty)$ with a Fatou norm. Suppose $f$ is a decreasing nonnegative function such that $f\in E_+\cap L_1$ and
$\Omega(f)=\mathcal Q_E(f).$  Then
$$ \lim_{\tau\to\infty}\tau^{-1}\|\sigma_\tau(f^*)\|_{E+L_\infty}=0.$$
\end{thm}

\begin{proof} An easy computation shows that $\mathcal Q(f+1)=\mathcal Q(f)+\mathcal Q(1).$  Hence
$\mathcal Q_{E+L_\infty}(f+1)\supset \mathcal Q_E(f)+\mathcal Q_{L_\infty}(1)=\Omega(f)+\Omega(1).$  If $0\le g\in \Omega(f+1)$ then $g-g\wedge 1\in \Omega(f)$ and $g\wedge 1 \in \Omega(1)$ so that $\Omega(f)+\Omega(1)=\Omega(f+1).$
Hence $\mathcal Q_{E+L_\infty}(f+1)=\Omega(f+1)$ and we can apply Theorem \ref{main1}.
\end{proof}

\section{The main results}
We can next state our main results:

\begin{thm}\label{Main3}  Let $E$ be a fully symmetric Banach function space on $(0,\infty)$ with a Fatou norm, and such that $E\setminus L_1\neq\emptyset.$   Suppose $f\in E.$  Then $\Omega(f)=\mathcal Q_E(f)$ if and only if
$\lim_{\tau\to \infty}\tau^{-1}\|\sigma_\tau(f^*)\|_E=0.$
\end{thm}

\begin{proof}  If $\lim_{\tau\to\infty}\tau^{-1}\|\sigma_\tau(f^*)\|_E=0$ then $\Omega_+(f)\subset \mathcal Q_E(f)$ by Theorem 25 of \cite{SukochevZanin2009}; thus $\Omega(f)=\mathcal Q_E(f).$  Conversely if $\Omega(f)=\mathcal Q_E(f)$ we have
either $$\lim_{\tau\to\infty}\tau^{-1}\|\sigma_\tau(f)\|_E=0$$ (when $f\notin L_1$ by Theorem \ref{main1}) or $$\lim_{\tau\to \infty}\tau^{-1}\|\sigma_\tau(f)\|_{E+L_\infty}=0$$ (when $f\in L_1$ by Theorem \ref{main2}).  Then Lemma \ref{L1} shows that in both cases we have $\lim_{\tau\to \infty}\tau^{-1}\|\sigma_\tau(f)\|_E=0$.\end{proof}

\begin{thm}\label{Main4} Let $E$ be a fully symmetric Banach function space on $(0,\infty)$ with a Fatou norm, and such that $E\subset L_1$.  If $f\in E$ then $\Omega(f)=\mathcal Q_E(f)$ if and only if
$\lim_{\tau\to \infty}\tau^{-1}\|\sigma_\tau(f^*)\|_{E+L_\infty}=\lim_{\tau\to \infty}\tau^{-1}\|\sigma_\tau(f^*)\chi_{(0,1)}\|_{E}=0.$
\end{thm}

\begin{proof}  The proof is very similar to that of Theorem \ref{Main3} using instead Theorem 24 of \cite{SukochevZanin2009}.
\end{proof}

We first give the extension to function spaces on $(0,1).$

\begin{thm}\label{Main30}  Let $E$ be a fully symmetric Banach function space on $(0,1)$ with a Fatou norm.    Suppose $f\in E.$  Then $\Omega(f)=\mathcal Q_E(f)$ if and only if
$\lim_{\tau\to \infty}\tau^{-1}\|\sigma_\tau(f^*)\|_E=0.$
\end{thm}

\begin{proof}  We define $F$ to be the function space on $(0,\infty)$ defined by $f\in F$ if and only if $f^*\chi_{(0,1)}\in E$ and $f\in L_1$ with the norm
$$\|f\|_F=\max(\|f^*\chi_{(0,1)}\|_E,\|f^*\|_{L_1}).$$
Suppose $f\in E$ is nonnegative and decreasing.  We will show that, regarding $f$ as an member of $F,$ we have $\Omega(f)=\mathcal Q_F(f).$  Note that the hypothesis $\Omega(f)=\mathcal Q_E(f)$ on $(0,1)$ implies only that if $g\in F$ and $g\in \Omega(f)$ then $g\in \mathcal Q_F(f)$ provided $\mu ({\rm supp}\, g)\le 1.$   We will show, however, that $\Omega(f)=\mathcal Q_F(f)$ and then the Theorem follows.

We will need the following Lemma:

\begin{lem}\label{needed}  Let $h\in F$ be nonnegative and decreasing.  Suppose $g\in F$ is nonnegative and decreasing and satisfies the conditions such that $g\preceq h$ and $g(x)=0$ for some $0<x<\infty.$  If there exists $c>0$ so that $g(t)\le h(t)$ for $0<t\le c$ then $g\in \mathcal Q_F(h).$
\end{lem}

\begin{proof}  For any $\theta>1$ we may pick $p>x/c$ so that
$$ \int_0^c h(s)\,ds \le \theta \int_{x/p}^c h(s)\,ds.$$
Then if $0<pa<b<\infty$ with $c\le b\le x$ we have
$$ \int_{pa}^b g(s)\,ds \le \int_0^b h(s)\,ds \le \theta \int_{x/p}^b h(s)\,ds \le \theta\int_{pa}^b h(s)\,ds.$$
The same inequality holds trivially if $b>x$ or $b<c.$  Thus by Theorem 6.3 of \cite{KaltonSukochev2008} we have
$g\in \lambda Q(h)$ for any $\lambda>1$ and the Lemma follows.\end{proof}

We continue the proof of the theorem.  We will assume, without loss of generality that $\int_0^1 f(t)\,dt=1.$  First suppose $g\in\Omega(f)$ is nonnegative, not identically zero, and decreasing and satisfies $g(x)=0$ for some $0<x<\infty.$  Given $\epsilon>0$ we may find $c_0>0$ so that
$$ \int_0^{c_0}f(s)\,ds <\epsilon/2.$$
Let $$\alpha =\sup_{0<t\le c_0}\frac{\int_0^t g(s)\,ds}{\int_0^t f(s)\,ds}.$$  We have $\alpha>0$ and we may pick $0<\beta<\alpha$ with $(\alpha-\beta)<\epsilon/2$ and then $0<c<c_0$ with
$$ g(c)>\beta f(c).$$
Let $c'\ge c$ be the least solution of
$$ \alpha\int_0^{c'}f(s)\,ds = \int_0^{c}g(s)\,d s+ (c'-c)g(c).$$

We now define
$$ h(t)=\begin{cases} g(t)+(1-\alpha) f(t), \qquad 0<t\le c\\
g(c)+(1-\alpha)f(t), \qquad c<t\le \min(c',1)\\
f(t), \qquad \qquad\qquad\qquad \min(c',1)<t\le 1\\
0 \qquad\qquad\qquad\qquad\qquad t\ge 1.\end{cases}$$

From the construction we have $h\in\Omega(f).$   Thus $h\in\mathcal Q_F(f).$ For any $t\le \min(c',1)$ we have
$$ \int_0^t g(s)\,ds \le \int_0^t h(s)\,ds.$$
If $c'<1$ then
$$ \int_0^t h(s)\,ds=\int_0^t f(s)\,ds \ge \int_0^t g(s)\,ds , \qquad t>c'.$$
If $c'\ge 1$ then
\begin{align*} \int_0^1 h(s)\,ds &\ge (1-\alpha)\int_0^1 f(s)\,ds +\int_0^c g(s)\,ds +g(c)(1-c)\\
&\ge (1-\alpha)\int_0^1 f(s)\,ds +\beta\int_0^c f(s)\,ds-\epsilon/2 +\beta f(c)(1-c)\\
&\ge (1-\alpha)\int_0^1 f(s)\,ds + \beta \int_0^1 f(s)\,ds -\epsilon/2\\
&\ge 1 -\epsilon.\end{align*}

Hence $(1-\epsilon)g\preceq h$ and by Lemma \ref{needed} we have $(1-\epsilon)g\in\mathcal Q_F(h)$.  Since $\epsilon>0$ is arbitrary we have $g\in \mathcal Q_F(h)\subset \mathcal Q_F(f).$

Finally let us note for general nonnegative decreasing $g\in F$ we have $\lim_{m\to\infty}\|g-g\chi_{(0,m)}\|_F=0$ so that we have $\Omega_F(f)=\mathcal Q_F(f).$

Now the result reduces to Theorems \ref{Main3} and \ref{Main4}.
\end{proof}

The extension to sequence spaces requires similar type of argument:

\begin{thm}\label{Main5}  Let $E$ be a fully symmetric Banach sequence space with a Fatou norm and such that
$E\setminus \ell_1\neq\emptyset$.    Suppose $\xi\in E.$  Then $\Omega(\xi)=\mathcal Q_E(\xi)$ if and only if
$\lim_{m\to \infty}m^{-1}\|\sigma_m(\xi^*)\|_E=0.$
\end{thm}

\begin{proof}  We consider the Banach function space $F$ of all bounded functions such that
$(f^*(0),f^*(1),\ldots)\in E$ with the norm
$$ \|f\|_F=f^*(0)+\|(a_n)_{n=1}^\infty\|_E,$$
where $f^*(0)=\|f\|_{L_\infty}$ and $a_n:=\int_{n-1}^nf^*(s)ds$,
$n\ge 1$. Then let $F(\mathbb N)$ be the subspace of $F$ of all
functions $f$ which are constant on each interval $(n-1,n]$.
Clearly, the Banach spaces $(F(\mathbb N), \|\cdot\|_F)$ and $(E,
\|\cdot\|_E)$ are linearly isomorphic, in particular
$$
\|\xi\|_E\leq \|\xi\|_F\leq 2 \|\xi\|_E, \quad \forall \xi\in
E=F(\mathbb N).
$$
Let $\mathbb E$ denote the conditional expectation operator
$$ \mathbb Ef =\sum_{n\in\mathbb N} \chi_{(n-1,n]}\int_{n-1}^nf(t)\,dt.$$

Suppose $\xi$ is a nonnegative decreasing sequence and let
$$f=\sum_{j=1}^{\infty}\xi_j\chi_{(j-1,j]}\in F.$$
The result will follow from:

\begin{thm}\label{equiv} $\Omega(\xi)=\mathcal Q_E(\xi)$ if and only if $\Omega(f)=\mathcal Q_F(f).$
\end{thm}

\begin{proof}  Let us suppose that $\Omega(\xi)=\mathcal Q_E(\xi)$. We may suppose $\xi$ has infinite support.
Suppose $g\in \Omega(f)$ is nonnegative and decreasing; we will show that $g\in \mathcal Q_F(f)$ and then it
follows that $\mathcal Q_F(f)=\Omega(f).$

Suppose $\epsilon>0$.  Then we may pick an integer $m\in\mathbb N$ so that
$$ g^*(m)-\lim_{n\to\infty}g^*(n)<\epsilon/4.$$
Now $\mathbb Eg\in \mathcal Q_F(f)$ since $\Omega(\xi)=\mathcal
Q_E(\xi)$.  Hence, by Proposition \ref{convexapprox} there is a
nonnegative decreasing function $h$ with $0\le h\le \mathbb Eg$
such that $\|\mathbb Eg-h\|_E<\epsilon/4$ and such that for some
$p\in\mathbb N,$
$$ \int_{pa}^b h(s)\,ds \le \int_a^b f(s)\,ds \qquad 0<pa<b<\infty.$$

Next we define
$$ \varphi(s)=\begin{cases} g(s), \qquad 0<s\le m\\ h(s), \qquad m<s<\infty.\end{cases}$$

Note that $0\le\varphi\preceq g\preceq f.$  We show that $\varphi\in \mathcal Q_F(f).$
Let us suppose $r>p$ and that $0<ra<b.$
Then if $m\le ra$ we clearly have
$$ \int_{ra}^b \varphi(s)\,ds \le \int_a^b f(s)\,ds.$$
On the other hand if $0<ra<m$, let $c=\min(b,m).$  Then
\begin{align*} \int_{ra}^b \varphi(s)\,ds &\le \int_0^b f(s)\,ds\\
&\le \int_a^b f(s)\,ds + c\xi_1/r\\
&\le \int_a^b f(s)\,ds + \frac{\xi_1}{(r-1)\xi_m}\int_{c/r}^c f(s)\,ds\\
&\le \left(1+\frac{\xi_1}{(r-1)\xi_m}\right)\int_a^b f(s)\,ds.\end{align*}

Since $r$ is arbitrary these estimate show that $\varphi\in \lambda \mathcal Q(f)$
for every $\lambda>1$ (Theorem 6.4 of \cite{KaltonSukochev2008}) and hence $\varphi\in\mathcal Q_F(f).$

Now
$$ \|g-\varphi\|_F=\|(g-\varphi)\chi_{(m,\infty)}\|_F\le \|(g-\mathbb Eg)\chi_{(m,\infty)}\|_F+\|\mathbb Eg-h\|_F.$$
However,
$$ \|(g-\mathbb Eg)\chi_{(m,\infty)}\|_F \le \sum_{j=m}^{\infty}2(g(j)-g(j+1))<\epsilon/2.$$  Hence
$$ d(g,\mathcal Q_E(f))\le \|g-\varphi\|_F<\epsilon.$$  Since $\epsilon>0$ is arbitrary we have $g\in\mathcal Q_F(f).$
This shows that $\Omega(f)=\mathcal Q_F(f).$

We next turn to the converse.  Assume $\mathcal Q_E(f)=\Omega(f)$
and that $\eta\in\Omega(\xi)$ is a decreasing sequence.  Let
$g=\sum_{n\in\mathbb N}\eta_n\chi_{(n-1,n]}$.  Then $g\in
\Omega(f)$ and so, by Proposition \ref{convexapprox}, given
$\epsilon>0$,  there exists a decreasing $0\le h\le g$ with
$\|g-h\|_F<\epsilon$ and such that for some $p\in\mathbb N$ we
have
$$ \int_{pa}^b h(s)\,ds\le \int_a^b f(s)\,ds, \qquad 0<pa<b<\infty.$$

Let $\zeta\in E$ be defined by $\zeta_n=\int_{n-1}^n h(s)\, ds.$  Then for $0\le pm\le n$ we have
$$ \sum_{k=pm+1}^n\zeta_k =\int_{pm}^nh(s)\,ds \le \int_m^n f(s)\,ds=\sum_{k=m+1}^n\xi_k.$$
Hence $\zeta \in \lambda \mathcal Q(\xi)$ for every $\lambda>1$, by Theorem 5.5 of \cite{KaltonSukochev2008},
so that $\zeta\in\mathcal Q_E(\xi).$
Furthermore,
$$ \|\eta-\zeta\|_E\leq \|g-\mathbb Eh\|_F \le \|g-h\|_F<\epsilon.$$
It now follows that $\eta\in\mathcal Q_E(\xi)$ and the proof of the Lemma is complete.
\end{proof}

Theorem \ref{Main5} now follows directly from Theorem \ref{equiv}.
\end{proof}

Let us observe that the argument of Theorem \ref{equiv}  allows us to complete the picture for positive orbits in \cite{SukochevZanin2009}:

\begin{thm} \label{SukZancomp}  Let $E$ be a fully symmetric sequence space with Fatou norm.  Then for any $\xi\in E_+$ the set
$\Omega_+(\xi)=\Omega(\xi)\cap E_+$ coincides with the closed convex hull of its extreme points if and only if
$$ \lim_{m\to\infty}m^{-1}\|\sigma_m(\xi)\|_E=0.$$
\end{thm}

In fact we can prove by the same argument as in Theorem \ref{equiv} that $\Omega_+(\xi)=\mathcal Q_E(\xi)\cap E_+$ if and only if $\Omega_+(f)=\mathcal Q_F(f)\cap F_+.$

We remark that in \cite{SukochevZanin2009} some examples of Marcinkiewicz spaces and Orlicz spaces are discussed in the context of Theorems \ref{Main3}, \ref{Main4}, \ref{Main30} and \ref{Main5}.  We refer the reader to \cite{SukochevZanin2009} for details.  We take the opportunity to improve Proposition 33 of \cite{SukochevZanin2009}:

\begin{prop}\label{improve} Let $M$ be an Orlicz function.  Then for any $f\in L_M(0,\infty)$ we have $\Omega(f)=\mathcal Q_{L_M}(f).$  Similarly for for any $\xi\in \ell_M$ we have $\Omega(\xi)=\mathcal Q_{\ell_M}(\xi).$
\end{prop}

\begin{proof}  We give the proof only for $L_M(0,\infty).$ Suppose first that $M(t)=o(t)$ when $t\to 0.$ We show that $\|\lim_{\tau\to\infty}\tau^{-1}\sigma_{\tau}f\|_{L_M}=0$ whenever $f\in (L_M)_+.$  Suppose $\alpha>0.$
$$ \int_0^{\infty}M\left(\frac{\alpha f(s/\tau)}\tau\right)ds=\int_0^{\infty}\tau M\left(\frac{\alpha f(s)}{\tau}\right)\,ds$$ for any $\tau>1.$ Since $f\in L_M$ there exists $\tau_0$ so that
$$\int_0^{\infty}\tau_0 M\left(\frac{\alpha f(s)}{\tau_0}\right)\,ds<\infty.$$  Now letting $\tau\to\infty$ we obtain from the Dominated Convergence Theorem that
$$\lim_{\tau\to\infty}\int_0^{\infty}\tau M\left(\frac{\alpha f(s)}{\tau}\right)\,ds=0$$
so that $\lim_{\tau\to\infty}\tau^{-1}\|\sigma_{\tau}f\|_{L_M}=0$ and we can apply Theorem \ref{Main3}.

Now if $M(t)\ge ct$ for all $t>0$ where $c>0$ we have $L_M\subset L_1.$  For any $\alpha>0$ and $\tau>1$
$$ \int_0^1 M\left(\frac{\alpha f^*(s/\tau)}\tau\right)ds=\int_0^{1}\tau\chi_{(0,\tau^{-1})}(s)M\left(\frac{\alpha f^*(s)}{\tau}\right)\,ds.$$  As before the right-hand side is integrable for some $\tau=\tau_0$ and we can apply the Dominated Convergence Theorem to deduce that $\tau^{-1}\|(\sigma_{\tau}f^*)\chi_{(0,1)}\|_{L_M}$ tends to $0$ as $\tau$ approaches infinity. Now one can apply Theorem \ref{Main4}.
\end{proof}

\section{A noncommutative analog}

Let $\mathcal H$ be a separable complex Hilbert space.  We denote by $\mathcal B(\mathcal H)$ the space of bounded operators on $\mathcal H$ and by $\mathcal K(\mathcal H)$ the ideal of compact operators on $\mathcal H.$  For any $T\in\mathcal B(\mathcal H)$ we define the singular values
$$s_n(T)=\inf \{\|T(I-P)\|,$$
where the infimum is taken over all orthogonal projections $P$ such that ${\rm rank}(P)<n\}.$

If $E$ is a symmetric sequence space then we can define a Banach ideal of compact operators on $\mathcal H$ by $T\in \mathcal S_E$ if and only if $(s_k(T))_{k=1}^{\infty}\in E$ and then the norm is given by $\|T\|_{E}=\|(s_k(T))_{k=1}^{\infty}\|_E.$  For fully symmetric spaces this is well-known (e.g. see \cite{GohbergKrein1969} but for symmetric spaces it follows from \cite{KaltonSukochev2008}).

Let $\mathcal H$ be a separable Hilbert space and suppose $T\in\mathcal K(\mathcal H).$  Let $\mathcal Q(T)$ be the convex hull  of the set $\{ATB;\ \|A\|,\|B\|\le 1\}.$  We define its orbit $\Omega(T)$ to be the closure of $\mathcal Q(T)$ in $\mathcal K(\mathcal H).$   It is easy to check from the definition that
$R\in \Omega(T)$ if and only if
$$ \sum_{k=1}^ns_k(R)\le \sum_{k=1}^ns_k(T), \qquad n=1,2,\ldots$$
For any symmetric Banach sequence space $E$ we may define $\mathcal Q_E(T)$ to be the closure of $\mathcal Q(T)$ in $\mathcal S_E.$

\begin{thm}\label{noncom} Let $\mathcal E$ be a fully symmetric sequence space with a Fatou norm.  Suppose $\mathcal S_E$ is the corresponding ideal of compact operators.  Then for $T\in\mathcal S_E$ we have
$\Omega(T)=\mathcal Q_E(T)$ if and only if $$\lim_{m\to\infty}m^{-1}\|\sigma_m(s_k(T))_{k=1}^{\infty}\|_E=0.$$
\end{thm}

\begin{proof} Let $\xi=(s_k(T))_{k=1}^{\infty}.$ Let $R\in\mathcal K(\mathcal H)$ and let $\eta=(s_k(R))_{k=1}^{\infty}$.
If $R\in \mathcal Q(T)$ then it follows from Proposition 8.6 and Theorem 5.5 of \cite{KaltonSukochev2008} that $\eta\in\lambda\mathcal Q(\xi)$ for every $\lambda>1.$

First suppose that $\Omega(T)=\mathcal Q_E(T).$  If $S\in \Omega(T)$ then given $\epsilon>0$ there exists $R\in\mathcal Q(T)$ with
$\|R-S\|_E<\epsilon.$  Let $\zeta=(s_k(S))_{k=1}^{\infty}.$  Then by the submajorization inequality of \cite{DoddsDoddsdePagter1989a},
$$ \eta-\zeta\preceq (s_k(R-S))_{k=1}^{\infty}$$ so that
$ \|\eta-\zeta\|_E <\epsilon.$  Since $\eta\in \mathcal Q_E(\xi)$ and $\epsilon>0$ is arbitrary, this implies that $\zeta\in\mathcal Q_E(\xi)$ and so $\mathcal Q_E(\xi)=\Omega(\xi).$  Theorem \ref{Main5} can then be applied.

The converse direction is immediate.
\end{proof}


\begin{bibsection}
\begin{biblist}

\bib{BennettSharpley1982}{book}{
  author={Bennett, C.},
  author={Sharpley, R.},
  title={Interpolation of operators},
  series={Pure and Applied Mathematics},
  volume={129},
  publisher={Academic Press Inc.},
  place={Boston, MA},
  date={1988},
}

\bib{BravermanMekler1977}{article}{
  author={Braverman, M. {\v {S}}.},
  author={Mekler, A. A.},
  title={The Hardy-Littlewood property for symmetric spaces},
  language={Russian},
  journal={Sibirsk. Mat. \v Z.},
  volume={18},
  date={1977},
  pages={522--540, 717},
}

\bib{Calderon1966}{article}{
  author={Calder{\'o}n, A.-P.},
  title={Spaces between $L\sp {1}$ and $L\sp {\infty }$ and the theorem of Marcinkiewicz},
  journal={Studia Math.},
  volume={26},
  date={1966},
  pages={273--299},
}

\bib{ChilinKryginSukochev1992}{article}{
  author={Chilin, V. I.},
  author={Krygin, A. V.},
  author={Sukochev, F. A.},
  title={Extreme points of convex fully symmetric sets of measurable operators},
  journal={Integral Equations Operator Theory},
  volume={15},
  date={1992},
  pages={186--226},
}

\bib{DoddsDoddsdePagter1989a}{article}{
  author={Dodds, P. G.},
  author={Dodds, T. K.-Y.},
  author={de Pagter, B.},
  title={A general Markus inequality},
  note={ {Miniconference on Operators in Analysis}, {Sydney}, {1989}, { Proc. Centre Math. Anal. Austral. Nat. Univ.}, {24}, {Austral. Nat. Univ.}, {Canberra} },
  date={1990},
  pages={47--57},
}

\bib{DoddsSukochevSchluchtermann2001}{article}{
  author={Dodds, P. G.},
  author={Sukochev, F. A.},
  author={Schl{\"u}chtermann, G.},
  title={Weak compactness criteria in symmetric spaces of measurable operators},
  journal={Math. Proc. Cambridge Philos. Soc.},
  volume={131},
  date={2001},
  pages={363--384},
}

\bib{GohbergKrein1969}{book}{
  author={Gohberg, I. C.},
  author={Kre{\u \i }n, M. G.},
  title={Introduction to the theory of linear nonselfadjoint operators},
  series={Translated from the Russian by A. Feinstein. Translations of Mathematical Monographs, Vol. 18},
  publisher={American Mathematical Society},
  place={Providence, R.I.},
  date={1969},
  pages={xv+378},
}

\bib{KaltonSukochev2008}{article}{
  author={Kalton, N. J.},
  author={Sukochev, F. A.},
  title={Symmetric norms and spaces of operators},
  journal={J. Reine Angew. Math.},
  volume={621},
  date={2008},
  pages={81--121},
}

\bib{KreinPetuninSemenov1982}{book}{
  author={Kre{\u \i }n, S. G.},
  author={Petun{\={\i }}n, Y. {\=I}.},
  author={Sem{\"e}nov, E. M.},
  title={Interpolation of linear operators},
  series={Translations of Mathematical Monographs},
  volume={54},
  publisher={American Mathematical Society},
  place={Providence, R.I.},
  date={1982},
}

\bib{LindenstraussTzafriri1979}{book}{
  author={Lindenstrauss, J.},
  author={Tzafriri, L.},
  title={Classical Banach spaces, II, Function spaces},
  series={Ergebnisse der Mathematik und ihrer Grenzgebiete [Results in
            Mathematics and Related Areas]},
  volume={97},
  publisher={Springer-Verlag},
  place={Berlin},
  date={1979},
}

\bib{Mitjagin1965}{article}{
  author={Mitjagin, B. S.},
  title={An interpolation theorem for modular spaces},
  language={Russian},
  journal={Mat. Sb. (N.S.)},
  volume={66 (108)},
  date={1965},
  pages={473--482},
}

\bib{Ryff1965}{article}{
  author={Ryff, J. V.},
  title={Orbits of $L\sp {1}$-functions under doubly stochastic transformations},
  journal={Trans. Amer. Math. Soc.},
  volume={117},
  date={1965},
  pages={92--100},
}

\bib{Ryff1967}{article}{
  author={Ryff, J. V.},
  title={Extreme points of some convex subsets of $L\sp {1}(0,\,1)$},
  journal={Proc. Amer. Math. Soc.},
  volume={18},
  date={1967},
  pages={1026--1034},
}

\bib{SukochevZanin2009}{article}{
  author={Sukochev, F. A.},
  author={Zanin, D.},
  title={Orbits in symmetric spaces},
  journal={J. Functional Analysis},
  volume={257},
  date={2009},
  pages={194\ndash 218},
}


\end{biblist}
\end{bibsection}
\end{document}